\documentclass[11pt]{article}

\usepackage[T1]{fontenc}
\usepackage{lmodern}
\usepackage[a4paper,margin=26mm]{geometry}
\usepackage{amsmath,amssymb,amsthm}
\usepackage{microtype}
\usepackage{needspace}
\usepackage{xcolor}

\definecolor{linkblue}{rgb}{0.05,0.18,0.45}

\usepackage[
  colorlinks=true,
  linkcolor=linkblue,
  citecolor=linkblue,
  urlcolor=linkblue
]{hyperref}

\hypersetup{
  pdftitle={Sharp Diameter Bounds for Nonnegative Cyclotomic Multiples},
  pdfauthor={Hu Tan and Ying Zhang},
  pdfsubject={Cyclotomic polynomials and integer tilings}
}

\newtheorem{theorem}{Theorem}[section]
\newtheorem{proposition}[theorem]{Proposition}
\newtheorem{lemma}[theorem]{Lemma}
\newtheorem{corollary}[theorem]{Corollary}

\theoremstyle{definition}
\newtheorem{example}[theorem]{Example}

\theoremstyle{remark}
\newtheorem{remark}[theorem]{Remark}

\numberwithin{equation}{section}

\DeclareMathOperator{\supp}{supp}
\DeclareMathOperator{\diam}{diam}
\DeclareMathOperator{\lcm}{lcm}

\newcommand{\T}{\mathbb T}
\newcommand{\Z}{\mathbb Z}
\newcommand{\R}{\mathbb R}
\newcommand{\E}{\mathbb E}
\newcommand{\NN}{\mathbb Z_{\ge0}}
\newcommand{\br}[2]{[\,#1\,]_{#2}}

\title{Sharp Diameter Bounds for Nonnegative\\Cyclotomic Multiples}

\author{%
\begin{tabular}[t]{c}
Hu Tan\textsuperscript{1}\\[3pt]
\small\textsuperscript{1}\,Academy of Mathematics and Systems Science\\
\small Chinese Academy of Sciences\\
\small Beijing 100190, China\\
\small\texttt{tanhu2020@amss.ac.cn}
\end{tabular}
\and
\begin{tabular}[t]{c}
Ying Zhang\textsuperscript{2}\\[3pt]
\small\textsuperscript{2}\,School of Mathematical Sciences\\
\small Soochow University\\
\small Suzhou 215006, China\\
\small\texttt{yzhang@suda.edu.cn}
\end{tabular}%
}

\date{}
\begin{document}
\maketitle
\begin{abstract}
Let \(N\ge2\) and let \(p\) be its least prime divisor. We prove that every nonzero polynomial with nonnegative real coefficients divisible by \(\Phi_N\) has support diameter at least \((p-1)N/p\). Equality holds precisely for positive scalar multiples of monomial shifts of the \(p\)-term geometric sum \(\sum_{j=0}^{p-1} X^{jN/p}\), thereby proving a conjecture of Steinberger. The proof turns cyclotomic divisibility into the vanishing of the first \(p-1\) Fourier moments of a positive measure on the circle and then applies a classical extremal trigonometric polynomial. As a consequence, we establish the Coven--Meyerowitz diameter bound under their tiling conditions and determine its equality cases. Longer initial intervals of vanishing Fourier coefficients yield stronger diameter bounds, including an explicit refinement in terms of the prime-power divisor sets. The extremal trigonometric polynomial also yields a quantitative concentration estimate for measures and cyclotomic multiples with near-minimal support diameter.

\end{abstract}

\noindent\textit{2020 Mathematics Subject Classification.}
Primary 11C08; Secondary 11B75, 05B45, 42A05.

\smallskip
\noindent\textit{Key words and phrases.}
Cyclotomic polynomials, nonnegative coefficients, integer tilings,
Fourier moments, diameter bounds.

\section{Introduction and main result}

The cyclotomic polynomial $\Phi_N$ has the primitive $N$th roots of unity
as its zeros, but its coefficients need not be nonnegative. A natural
extremal question asks for the least degree of a nonzero
multiple with nonnegative coefficients. If $p$ is the least prime
divisor of $N$, the geometric sum
\begin{equation}\label{eq:fiber}
 G_{N,p}(X)=1+X^{N/p}+\cdots+X^{(p-1)N/p}
\end{equation}
is such a multiple. Steinberger conjectured that it is the unique monic
multiple of minimum degree \cite[Conjecture~1]{Steinberger}.

For a nonzero polynomial $F(X)=\sum_jc_jX^j$, write
\[
 \supp F=\{j:c_j\ne0\},\qquad
 D(F)=\max\supp F-\min\supp F.
\]
Our main result proves this conjecture in the translation-invariant
form that will be useful for integer tilings.

\Needspace{12\baselineskip}
\begin{theorem}\label{thm:main}
Let $N\ge2$ and let $p$ be its least prime divisor. If
$F\in\R[X]\setminus\{0\}$ has nonnegative coefficients and $\Phi_N\mid F$,
then
\begin{equation}\label{eq:main}
 D(F)\ge\frac{p-1}{p}N.
\end{equation}
Equality holds if and only if
\begin{equation}\label{eq:main-equality}
 F(X)=cX^bG_{N,p}(X),\qquad c>0,\quad b\in\NN.
\end{equation}
In particular, $G_{N,p}$ is the unique monic polynomial of least degree
with nonnegative real coefficients divisible by $\Phi_N$.
\end{theorem}

Steinberger proved the conjecture when $N$ is even, when $N$ is a prime
power, or when its distinct prime factors $p=p_1<p_2<\cdots<p_s$ satisfy
\begin{equation}\label{eq:old-condition}
 \frac2{p_1}>\sum_{i=2}^s\frac1{p_i}
\end{equation}
\cite[Theorem~1]{Steinberger}. This includes every $N$ with at most
three distinct prime factors. Theorem~\ref{thm:main} removes the
restriction \eqref{eq:old-condition}.

The connection with integer tilings concerns a specific period. For a
finite set $A\subset\Z$ with $|A|\ge2$, translate $A$ into $\NN$ and define
\begin{equation}\label{eq:SA}
 A(X)=\sum_{a\in A}X^a,\qquad
 S_A=\{q^e:q\text{ prime},\ e\ge1,\ \Phi_{q^e}\mid A(X)\}.
\end{equation}
The Coven--Meyerowitz conditions are
\begin{align*}
 \textnormal{(T1)}\quad & |A|=\prod_{u\in S_A}\Phi_u(1),\\
 \textnormal{(T2)}\quad & \Phi_{u_1\cdots u_k}\mid A(X)
 \quad\text{if }u_1,\ldots,u_k\in S_A
 \text{ are powers of distinct primes}.
\end{align*}
They imply that $A$ tiles $\Z$ with period $M=\lcm(S_A)$
\cite[Theorem~A]{CM}. Since (T2) gives $\Phi_M\mid A(X)$,
Theorem~\ref{thm:main} implies the following result.

\begin{corollary}\label{cor:CM}
Let $A\subset\Z$ be finite with $|A|\ge2$, and suppose that $A$ satisfies
(T1) and (T2). Set $M=\lcm(S_A)$, and let $p$ be the least prime divisor
of $|A|$. Then
\begin{equation}\label{eq:CM}
 \diam(A)\ge\frac{p-1}{p}M.
\end{equation}
Equality holds precisely when, for some $n\ge1$ and $b\in\Z$,
\begin{equation}\label{eq:CM-equality}
 M=p^n,\qquad A=b+p^{n-1}\{0,1,\ldots,p-1\}.
\end{equation}
\end{corollary}

Coven and Meyerowitz asserted both \eqref{eq:CM} and its equality case
in the remarks following \cite[Lemma~2.1]{CM}, without the hypotheses
(T1) and (T2). \L aba and Zakharov explain why the unrestricted assertion
fails and ask whether the inequality holds under these hypotheses
\cite[Section~4, equation~(14)]{LZ}. Corollary~\ref{cor:CM} proves the
latter statement, including the asserted equality case. It does not require,
or prove, the separate assertion that every finite integer tile
satisfies (T2).

The proof of Theorem~\ref{thm:main} uses consecutive Fourier moments.
Because every integer $1\le k<p$ is coprime to $N$, cyclotomic
divisibility gives
\begin{equation}\label{eq:intro-moments}
 F(e^{2\pi ik/N})=0\qquad(1\le k<p).
\end{equation}
After normalization, the coefficients of $F$ therefore form a positive
measure on the circle with its first $p-1$ moments equal to zero.
Such a measure cannot leave an open arc of length greater than
$2\pi/p$ empty. The precise support bound and its equality case follow
from the sine vector of a finite tridiagonal matrix.

This analytic ingredient has a classical background. The extremal
polynomial is the one associated with Fej\'er's coefficient inequality
for nonnegative trigonometric polynomials \cite{Fejer}. For finitely
supported probability measures, the moment condition says that the
associated positive quadrature rule is exact for trigonometric
polynomials of the prescribed degree with respect to Haar measure;
see \cite[Section~1]{Peherstorfer}. For equal weights it defines a
spherical design on the circle, and the gap bound is the circle case
of the covering bound of Fazekas and Levenshtein \cite[Theorem~2]{FL}.
We include an elementary proof for arbitrary positive weights, together
with the equality case. For cyclotomic multiples, the required
vanishing moments follow from \eqref{eq:intro-moments}.

The rest of this paper is organized as follows. 
Section~\ref{sec:circle} proves the support bound, and
Section~\ref{sec:cyclotomic} proves the main theorem and a refinement
using the full set of cyclotomic divisors. Section~\ref{sec:tilings}
returns to integer tilings and extracts a stronger bound from (T2).
Section~\ref{sec:stability} gives a quantitative version of the equality
case, including an estimate in the original exponent scale.

\section{Positive measures with vanishing moments}\label{sec:circle}

We identify the circle with $\T=\R/(2\pi\Z)$ and use radians for angular
distance. For a finite measure $\mu$, put
\[
 \widehat\mu(k)=\int e^{ik\theta}\,d\mu(\theta).
\]
If $\mu$ is positive, then
$\widehat\mu(-k)=\overline{\widehat\mu(k)}$.

\begin{lemma}[Arc bound]\label{lem:arc}
Let $r\ge2$ be an integer. Suppose that a nonzero finite positive measure
$\mu$ is supported on $[-\alpha,\alpha]$, where $0\le\alpha<\pi$, and
\begin{equation}\label{eq:moments}
 \widehat\mu(k)=0\qquad(1\le k\le r-1).
\end{equation}
Then
\begin{equation}\label{eq:arc-bound}
 \alpha\ge\pi-\frac\pi r.
\end{equation}
If equality holds, $\mu$ assigns equal positive masses to the $r$ points
\begin{equation}\label{eq:nodes}
 -\alpha+\frac{2\pi j}{r},\qquad 0\le j\le r-1,
\end{equation}
and has no other support.
\end{lemma}

\begin{proof}
Normalize $\mu$ to be a probability measure, and write
$\E f=\int f\,d\mu$. Set
\begin{equation}\label{eq:sine-vector}
 t=\frac\pi r,\qquad
 v_j=(-1)^j\sin((j+1)t)\quad(0\le j\le r-2),
\end{equation}
and define
\begin{equation}\label{eq:Q}
 Q(z)=\sum_{j=0}^{r-2}v_jz^j,\qquad
 S=\sum_{j=0}^{r-2}v_j^2=\frac r2.
\end{equation}
The last equality is the elementary sum
$\sum_{j=1}^{r-1}\sin^2(j\pi/r)=r/2$.
The vanished moments give
\begin{align}
 \E|Q(e^{i\theta})|^2&=S,\label{eq:norm}\\
 \E\!\left[\cos\theta\,|Q(e^{i\theta})|^2\right]
 &=\sum_{j=0}^{r-3}v_jv_{j+1}
 =-\cos(t)S.\label{eq:cosine}
\end{align}
Indeed, the trigonometric polynomials in these two expectations have
degrees at most $r-2$ and $r-1$, respectively, so only their constant
Fourier coefficients remain. The final identity follows by multiplying
\[
 v_{j-1}+v_{j+1}=-2\cos(t)v_j\quad(0\le j\le r-2),
 \qquad v_{-1}=v_{r-1}=0,
\]
by $v_j$ and summing. This also covers $r=2$, when the adjacent sum
is empty.

On the support of $\mu$, $\cos\theta\ge\cos\alpha$. Hence
\[
 -\cos(t)S\ge\cos\alpha\,S.
\]
Since $S>0$ and cosine is decreasing on $[0,\pi]$, this proves
\eqref{eq:arc-bound}.

If $\alpha=\pi-t$, equality in this integral inequality forces the
support into the zero set of
\begin{equation}\label{eq:H}
 (\cos\theta+\cos t)|Q(e^{i\theta})|^2
 \quad\text{on }[-\alpha,\alpha].
\end{equation}
The recurrence above gives the polynomial identity
\begin{equation}\label{eq:Qidentity}
 (1+2\cos(t)z+z^2)Q(z)=\sin(t)\bigl(1+(-z)^r\bigr).
\end{equation}
The polynomial $1+(-z)^r$ has the $r$ distinct roots
$e^{i(-\alpha+2\pi j/r)}$. The two roots of the quadratic on the left
are $e^{\pm i\alpha}$, and the remaining $r-2$ roots are the roots of
$Q$. Thus the zeros in \eqref{eq:H} are exactly \eqref{eq:nodes}.
If $w_j$ denotes the mass at the $j$th node, allowing initially $w_j=0$,
then
\[
 \sum_{j=0}^{r-1}w_je^{2\pi ikj/r}=0\quad(1\le k<r),
 \qquad \sum_{j=0}^{r-1}w_j=1.
\]
Discrete Fourier inversion gives $w_j=1/r$ for all $j$.
\end{proof}

\begin{remark}\label{rem:quadrature}
For a probability measure satisfying \eqref{eq:moments}, integration
agrees with normalized Haar integration on every trigonometric
polynomial of degree at most $r-1$. In particular, the two integrals
in \eqref{eq:norm}--\eqref{eq:cosine} equal the corresponding Haar integrals.
In matrix terms, the compression of multiplication by $\cos\theta$
to the orthonormal family $1,z,\ldots,z^{r-2}$ is the tridiagonal matrix
with $1/2$ on its adjacent diagonals and zero elsewhere. Its least
eigenvalue is $-\cos(\pi/r)$, with eigenvector \eqref{eq:sine-vector}.
\end{remark}

\begin{corollary}[Circular gaps]\label{cor:circle-gap}
If a nonzero finite positive measure on $\T$ satisfies
\eqref{eq:moments}, every open arc disjoint from its support has length
at most $2\pi/r$. If such an arc has length $2\pi/r$, the measure is a
positive scalar multiple of the uniform measure on a rotated regular
$r$-gon.
\end{corollary}

\begin{proof}
Rotate the complementary closed arc to have midpoint zero and apply
Lemma~\ref{lem:arc}. Rotating a measure preserves the vanishing of each
Fourier moment.
\end{proof}

\section{Cyclotomic divisibility and the first nonzero moment}
\label{sec:cyclotomic}

\begin{proof}[Proof of Theorem~\ref{thm:main}]
Removing the initial monomial factor preserves divisibility by
$\Phi_N$ and the support diameter. Write the resulting polynomial as
$F_0(X)=\sum_{a=0}^{D}c_aX^a$, where $c_a\ge0$ and $c_0c_D>0$.
If $D\ge N$, the desired inequality is strict. Otherwise set
\begin{equation}\label{eq:polynomial-measure}
 \alpha=\frac{\pi D}{N},\qquad
 \theta_a=\frac{2\pi a}{N}-\alpha,\qquad
 \mu=\frac1{F_0(1)}\sum_{a=0}^Dc_a\delta_{\theta_a}.
\end{equation}
For $1\le k<p$, the integer $k$ is coprime to $N$, so
\[
 \widehat\mu(k)
 =\frac{e^{-ik\alpha}}{F_0(1)}F_0(e^{2\pi ik/N})=0.
\]
Lemma~\ref{lem:arc}, applied with $r=p$, gives
$\pi D/N\ge\pi-\pi/p$, proving \eqref{eq:main}.

In the equality case, the points \eqref{eq:nodes} correspond exactly to
the exponents $0,N/p,\ldots,(p-1)N/p$, and their coefficients are equal.
Restoring the initial monomial gives \eqref{eq:main-equality}.
Conversely, $G_{N,p}$ vanishes at every primitive $N$th root, since the
ratio in its geometric sum has order $p$. It is therefore divisible by
$\Phi_N$ and has the asserted diameter. Minimum degree requires $b=0$,
and monicity then requires $c=1$.
\end{proof}

\begin{proposition}[A longer interval of zero moments]\label{prop:block}
Let $N\ge2$, $2\le r\le N$, and let $F\ne0$ have nonnegative real
coefficients. If
\begin{equation}\label{eq:block}
 F(e^{2\pi ik/N})=0\qquad(1\le k<r),
\end{equation}
then $D(F)\ge N(1-1/r)$. Equality is possible only if $r\mid N$, and
then it holds exactly for
\begin{equation}\label{eq:block-equality}
 F(X)=cX^b\sum_{j=0}^{r-1}X^{jN/r},\qquad c>0,\quad b\in\NN.
\end{equation}
\end{proposition}

\begin{proof}
If $D(F)\ge N$, the inequality is strict. Otherwise remove the initial
monomial factor and use \eqref{eq:polynomial-measure} and
Lemma~\ref{lem:arc} with the given
$r$. At equality the exponents, after subtracting the smallest one,
are $jN/r$ for $0\le j<r$. In particular, $N/r$ must be an integer.
The converse follows by summing a geometric progression.
\end{proof}

For rational coefficients, this strengthening can be read directly from
the cyclotomic divisor set. Put $\zeta_N=e^{2\pi i/N}$ and define
\begin{equation}\label{eq:rho}
 \rho_N(F)=\min\{k\ge1:F(\zeta_N^k)\ne0\}.
\end{equation}
This minimum is at most $N$, because $F(1)>0$.

\begin{corollary}\label{cor:rho}
Let $N\ge2$ and let $F\in\mathbb Q[X]\setminus\{0\}$ have nonnegative coefficients.
Then
\begin{equation}\label{eq:rho-divisor}
 \rho_N(F)=\min\{d:d\mid N,\ \Phi_{N/d}\nmid F\},
\end{equation}
and
\begin{equation}\label{eq:rho-bound}
 D(F)\ge N-\frac{N}{\rho_N(F)}.
\end{equation}
Equality holds precisely for the polynomials in
\eqref{eq:block-equality} with $r=\rho_N(F)$; when $r=1$ this means
a positive monomial.
\end{corollary}

\begin{proof}
Let $k=\rho_N(F)$ and $d=\gcd(k,N)$. The numbers $\zeta_N^k$ and
$\zeta_N^d$ are primitive roots of the same order $N/d$.
Rationality and irreducibility of $\Phi_{N/d}$ imply that $F$ vanishes
at one if and only if it vanishes at the other. Thus
$F(\zeta_N^d)\ne0$, and minimality gives $k=d$.
This also proves \eqref{eq:rho-divisor}, with $\Phi_1(X)=X-1$.
For $r\ge2$, apply Proposition~\ref{prop:block}; for $r=1$, the
assertion is $D(F)\ge0$ with its evident equality case.
\end{proof}

\begin{remark}\label{rem:real}
The rationality assumption in \eqref{eq:rho-divisor} is essential for
the conjugacy argument. For real coefficients, a zero at one primitive
root need not imply divisibility by its cyclotomic polynomial. For
example, with $\varphi=(1+\sqrt5)/2$, the polynomial
$\varphi+X^2+X^3$ vanishes at $e^{2\pi i/5}$ but has diameter $3<4$.
It is not divisible by $\Phi_5$. Theorem~\ref{thm:main} assumes actual
cyclotomic divisibility and therefore applies to real coefficients.
\end{remark}

\section{Diameter bounds for integer tilings}\label{sec:tilings}

We first prove Corollary~\ref{cor:CM} and identify the role of its
period. We then use all the prime-power data in (T2) to obtain a
stronger bound.

\begin{proof}[Proof of Corollary~\ref{cor:CM}]
Write $M=\prod_{i=1}^sp_i^{n_i}$. Each maximal prime power $p_i^{n_i}$
belongs to $S_A$, so (T2) gives $\Phi_M\mid A(X)$.
Since $\Phi_{q^e}(1)=q$, (T1) shows that $M$ and $|A|$ have exactly
the same prime divisors. Theorem~\ref{thm:main} proves \eqref{eq:CM}.

At equality, the zero--one coefficients give
$A=b+\{0,M/p,\ldots,(p-1)M/p\}$, so $|A|=p$.
Condition (T1) now implies that $S_A$ consists of a single prime power
$p^n$. Thus $M=p^n$ and \eqref{eq:CM-equality} follows.
Conversely, the normalized mask of the stated set is $\Phi_{p^n}$,
which satisfies (T1) and (T2) and attains equality.
\end{proof}

In particular, if $M$ has at least two distinct prime divisors, the
integer-valued bound is strict:
\begin{equation}\label{eq:strict}
 \diam(A)\ge\frac{p-1}{p}M+1.
\end{equation}

\subsection{The canonical period}

Let $\mathcal P(A)$ denote the smallest positive period among all
tilings of $\Z$ by translates of $A$. Under (T1) and (T2),
\begin{equation}\label{eq:minperiod}
 \mathcal P(A)=\lcm(S_A)=M.
\end{equation}
This follows from the construction and the period observation of
Coven and Meyerowitz \cite[Theorem~A and remarks after Lemma~2.1]{CM}.
For completeness, suppose $A\oplus B=\Z/K\Z$ is a tiling modulo $K$,
with finite representatives $B$. Then $A(1)B(1)=K$, and every
$\Phi_u$ with $u\mid K$, $u>1$, divides $A(X)B(X)$.
The product of the prime-power factors among these has value $K$ at
$1$. If some $q^e\in S_A$ did not divide $K$, including the additional
distinct factor $\Phi_{q^e}$ would force $qK$ to divide $A(1)B(1)=K$,
a contradiction. Hence $M\mid K$. The construction supplies a tiling
with period $M$, proving \eqref{eq:minperiod}.
Consequently Corollary~\ref{cor:CM} yields
\[
 \mathcal P(A)\le\frac{p}{p-1}\diam(A).
\]

\begin{example}[The need for mixed divisibility]\label{ex:no-T2}
The zero--one polynomial
\[
 A(X)=(1+X^3+X^6)(1+X^5+X^{10}+X^{15}+X^{20})
      =\Phi_9(X)\Phi_{25}(X)
\]
has $S_A=\{9,25\}$, $|A|=15$, and $M=225$, but $\diam(A)=26<150$.
It satisfies (T1) and fails (T2). This is a concrete instance of the
unrestricted counterexamples discussed in \cite[Section~4]{LZ}; it
explains why the definition of $M$ alone does not imply the diameter
bound.
\end{example}

\subsection{A refinement from the prime-power divisor sets}

For each prime $q\mid M$, write $n_q=v_q(M)$ and
\[
 E_q=\{e:q^e\in S_A\},\qquad
 \ell_q=\max\{h:\{n_q-h+1,\ldots,n_q\}\subseteq E_q\}.
\]
Thus $\ell_q$ is the number of consecutive supported levels at the top,
and $1\le\ell_q\le n_q$. Define
\begin{equation}\label{eq:L}
 L_A=
 \begin{cases}
 M,&\ell_q=n_q\text{ for every }q\mid M,\\
 \displaystyle\min_{\substack{q\mid M\\\ell_q<n_q}}q^{\ell_q},
 &\text{otherwise}.
 \end{cases}
\end{equation}
In both cases $L_A\mid M$.

\begin{theorem}\label{thm:top}
Let $A\subset\Z$ be finite and nonempty, suppose $S_A\ne\varnothing$, and assume
(T2). With $M=\lcm(S_A)$ and $L_A$ as in \eqref{eq:L},
\begin{equation}\label{eq:top}
 \diam(A)\ge M-\frac{M}{L_A}.
\end{equation}
If every $\ell_q=n_q$, equality holds precisely for translates of
$\{0,1,\ldots,M-1\}$. Otherwise equality holds precisely when
\begin{equation}\label{eq:top-equality}
 M=q^n,\qquad
 A=b+q^{n-\ell}\{0,1,\ldots,q^\ell-1\},
 \qquad 1\le\ell<n,
\end{equation}
in which case $L_A=q^\ell$.
\end{theorem}

\begin{proof}
Fix $1\le k<L_A$. The order of $e^{2\pi ik/M}$ is
\[
 \frac{M}{\gcd(k,M)}
 =\prod_{q\mid M}q^{\max\{n_q-v_q(k),0\}}.
\]
If $\ell_q=n_q$, every positive exponent in this expression belongs to
$E_q$. If $\ell_q<n_q$, then $k<L_A\le q^{\ell_q}$, so
$v_q(k)<\ell_q$ and
$n_q-v_q(k)\in\{n_q-\ell_q+1,\ldots,n_q\}\subseteq E_q$.
Since $k<M$, the order is greater than one. Condition (T2) therefore
gives $A(e^{2\pi ik/M})=0$. Proposition~\ref{prop:block} proves
\eqref{eq:top} and shows that equality requires
\begin{equation}\label{eq:Lcycle}
 A(X)=X^b\sum_{j=0}^{L_A-1}X^{jM/L_A},
\end{equation}
after a nonnegative translation.

If $L_A=M$, this is a full interval, which has the stated divisor set
and attains equality. In the other case $L_A=q^\ell$ for some prime
$q$, with $\ell<n_q$. The geometric sum in \eqref{eq:Lcycle} is
\[
 \frac{X^M-1}{X^{M/q^\ell}-1}.
\]
Its prime-power cyclotomic divisors are exactly
$q^{n_q-\ell+1},\ldots,q^{n_q}$: for every other prime, the valuations
in $M$ and $M/q^\ell$ coincide. The definition $M=\lcm(S_A)$ consequently
forces $M=q^{n_q}$. This proves necessity in \eqref{eq:top-equality}.
Conversely, the sets in \eqref{eq:top-equality} have exactly that
consecutive divisor set, satisfy (T2), and attain the bound.
\end{proof}

If some $\ell_q<n_q$, choose $q$ with $L_A=q^{\ell_q}$.
The order $M/L_A$ contains the first missing exponent
$n_q-\ell_q$ at that prime. Thus this order is outside the list of
orders supplied by the products in (T2). Additional cyclotomic
divisors of $A(X)$ may still give a longer interval of zero moments;
Corollary~\ref{cor:rho} captures those additional zeros.

\begin{example}\label{ex:450}
Use the notation $\br mY=1+Y+\cdots+Y^{m-1}$, and consider
\[
 A(X)=\br2{X^{225}}\,\br3{X^{75}}\,\br5{X^5}.
\]
Its $30$ exponents are $225u+75v+5w$, with $0\le u<2$, $0\le v<3$,
and $0\le w<5$. They are distinct: after division by $5$, the residue
modulo $15$ determines $w$, and then $3u+v$ determines $u,v$.
The three factors have respective cyclotomic divisor sets
\[
 \{2,6,10,18,30,50,90,150,450\},\qquad
 \{9,45,225\},\qquad \{25\}.
\]
Hence $S_A=\{2,9,25\}$ and both (T1) and (T2) hold.
Here $M=450$, $L_A=3$, and $\diam(A)=395$.
Theorem~\ref{thm:top} gives the lower bound $300$, compared with $225$
from Corollary~\ref{cor:CM}. The product of all factors required by
(T2) has degree
\[
 (1+\varphi(2))(1+\varphi(9))(1+\varphi(25))-1=293,
\]
so the refinement also exceeds the direct degree estimate.
Here $\varphi$ denotes Euler's totient function. The additional mixed
divisors in this example give $\rho_{450}(A)=6$, so
Corollary~\ref{cor:rho} improves the bound further to $375$.
\end{example}

\begin{example}\label{ex:36}
For $A(X)=\br6{X^6}$ one has $S_A=\{4,9\}$ and $M=36$.
Theorem~\ref{thm:top} gives $L_A=2$, but direct geometric summation gives
$\rho_{36}(A)=6$. Thus Corollary~\ref{cor:rho} gives the exact bound
$D(A)=36-36/6=30$. The additional divisors $\Phi_{12}$ and $\Phi_{18}$
account for the longer interval of vanishing moments.
\end{example}

\section{Quantitative rigidity near equality}\label{sec:stability}

The sine polynomial also controls how far a nearly extremal measure
can lie from the regular polygon in Lemma~\ref{lem:arc}.
For $V\subset\T$, let $d_{\T}(\theta,V)$ be the shortest angular
distance from $\theta$ to $V$.

\begin{proposition}\label{prop:stability}
Let $r\ge2$, and let $\mu$ be a probability measure satisfying
$\widehat\mu(k)=0$ for $1\le k<r$. Suppose that
$\supp\mu\subset[-\alpha,\alpha]$, where
$\alpha_*=\pi-\pi/r\le\alpha\le\pi$. Set
\[
 V_r=\left\{-\alpha_*+\frac{2\pi j}{r}:0\le j<r\right\}\pmod{2\pi}.
\]
Then
\begin{equation}\label{eq:stability}
 \int d_{\T}(\theta,V_r)^2\,d\mu(\theta)
 \le\frac{\pi^2[-\cos(\pi/r)-\cos\alpha]}
 {2r[1-\cos(\pi/r)]}
 \le\pi(\alpha-\alpha_*).
\end{equation}
In particular, as $\alpha$ decreases to $\alpha_*$, such measures
converge weakly to the uniform probability measure on $V_r$.
\end{proposition}

\begin{proof}
Keep $t=\pi/r$, $Q$, and $S=r/2$ from
\eqref{eq:sine-vector}--\eqref{eq:Q}. Write
$c=\cos t$, $s=\sin t$, $\delta=-c-\cos\alpha\ge0$, and
$x(\theta)=\cos\theta+c$. The moment identities give
\begin{equation}\label{eq:stability-moments}
 \int |Q|^2\,d\mu=S,\qquad
 \int x|Q|^2\,d\mu=0.
\end{equation}
On the supporting arc, $-\delta\le x\le1+c$, so
\[
 x^2\le(1+c-\delta)x+\delta(1+c).
\]
Multiply by $|Q|^2$ and integrate to obtain
\begin{equation}\label{eq:weighted-square}
 \int x^2|Q|^2\,d\mu\le\delta(1+c)S.
\end{equation}
Taking squared moduli in \eqref{eq:Qidentity} gives the pointwise identity
\[
 x(\theta)^2|Q(e^{i\theta})|^2
 =\frac{s^2}{2}\bigl(1+(-1)^r\cos(r\theta)\bigr).
\]
If $d=d_{\T}(\theta,V_r)$, then $0\le d\le\pi/r$ and
\[
 1+(-1)^r\cos(r\theta)
 =2\sin^2(rd/2)\ge\frac{2r^2}{\pi^2}d^2.
\]
Together with \eqref{eq:weighted-square} and
$s^2=(1-c)(1+c)$, this proves the first inequality in
\eqref{eq:stability}. Since $\alpha_*\ge\pi/2$,
\[
 \delta=\int_{\alpha_*}^{\alpha}\sin u\,du
 \le s(\alpha-\alpha_*).
\]
The second inequality follows from
$s/(1-c)=\cot(t/2)\le2r/\pi$.

For the convergence assertion, probability measures on the circle are
weakly sequentially compact. The estimate forces every subsequential
limit to be supported on $V_r$. Its first $r-1$ moments still vanish,
and Fourier inversion forces equal weights. Every subsequential limit
is therefore the same uniform measure.
\end{proof}

The estimate has a simple form in the exponent variable. For a real
set $\mathcal G$ considered modulo $N$, write
\[
 d_N(a,\mathcal G)=\min_{g\in\mathcal G,\,m\in\Z}|a-g+mN|.
\]

\begin{corollary}\label{cor:exponent-stability}
Under the hypotheses of Theorem~\ref{thm:main}, write
$F(X)=\sum_ac_aX^a$, $b=\min\supp F$, and
\[
 D(F)=\frac{p-1}{p}N+\Delta\le N.
\]
For the real grid
\[
 \mathcal G=\left\{b+\frac\Delta2+\frac{jN}{p}:0\le j<p\right\}
 \pmod N,
\]
one has
\begin{equation}\label{eq:exponent-stability}
 \frac1{F(1)}\sum_ac_a\,d_N(a,\mathcal G)^2\le\frac{N\Delta}{4}.
\end{equation}
Consequently, for every $h>0$, the fraction of coefficient mass at
distance at least $h$ from $\mathcal G$ is at most $N\Delta/(4h^2)$.
\end{corollary}

\begin{proof}
Apply Proposition~\ref{prop:stability} to the coefficient measure
\eqref{eq:polynomial-measure}, with $r=p$; the same formula remains
valid for $D(F)=N$.
Its angular excess is $\alpha-\alpha_*=\pi\Delta/N$.
The inverse change of variables sends $V_p$ to $\mathcal G$, and
multiplies angular distances by $N/(2\pi)$. Thus
\[
 \frac1{F(1)}\sum_ac_a\,d_N(a,\mathcal G)^2
 \le\frac{N^2}{4\pi^2}\,\pi\frac{\pi\Delta}{N}
 =\frac{N\Delta}{4}.
\]
The final assertion is Markov's inequality.
\end{proof}

\section*{Acknowledgments}

The authors thank the authors of the works cited in this paper for the
questions and ideas that motivated the present study. OpenAI's GPT Astra
was used during the preparation of the manuscript for language editing
and as a research aid in exploring and identifying useful examples.
All mathematical statements, constructions, and proofs were independently
verified by the authors, who take full responsibility for the contents
of the paper.

\end{document}